\documentclass[twoside]{amsart}
\usepackage{latexsym}
\usepackage{amssymb,amsmath,amsopn}
\usepackage[dvips]{graphicx}   
\usepackage{color,epsfig}      

\newtheorem{thm}{Theorem}
 
\newtheorem{lem}{Lemma}
\newtheorem{cor}{Corollary}

\theoremstyle{definition}
\newtheorem{defn}{Definition} 
\newtheorem{rem}{Remark}

\newtheorem*{conj}{Conjecture}
\newtheorem{prob}{Problem}

\renewcommand{\Re}{\mathbb R}

\newcommand{\B}{\mathbf B}

\renewcommand{\S}{\mathbb{S}}

\DeclareMathOperator{\bd}{bd}

\DeclareMathOperator{\dist}{dist}
\DeclareMathOperator{\conv}{conv}
\DeclareMathOperator{\card}{card}

\DeclareMathOperator{\diam}{diam}

\DeclareMathOperator{\Sym}{Sym}

\begin{document}
\title[On multiple Borsuk numbers]{On the multiple Borsuk numbers of sets}
\author[M. Hujter and Z. L\'angi]{Mih\'aly Hujter and Zsolt L\'angi}
\address{Mih\'aly Hujter, Dept.\ of Differential Equations, Budapest
University of Technology and Economics, Budapest, Egry J\'ozsef u. 1.,
Hungary, 1111}
\email{hujter@math.bme.hu}
\address{Zsolt L\'angi, Dept.\ of Geometry, Budapest University of
Technology and Economics, Budapest, Egry J\'ozsef u. 1., Hungary, 1111}
\email{zlangi@math.bme.hu}

\begin{abstract}
The \emph{Borsuk number} of a set $S$ of diameter $d >0$ in Euclidean $n$%
-space is the smallest value of $m$ such that $S$ can be partitioned into $m$
sets of diameters less than $d$. Our aim is to generalize this notion in the
following way: The \emph{$k$-fold Borsuk number} of such a set $S$ is the
smallest value of $m$ such that there is a $k$-fold cover of $S$ with $m$ sets of diameters
less than $d$. In this paper we characterize
the $k$-fold Borsuk numbers of sets in the Euclidean plane, give
bounds for those of centrally symmetric sets, smooth bodies and
convex bodies of constant width, and examine them for finite
point sets in the Euclidean $3$-space.
\end{abstract}

\maketitle

\section{Introduction}

In 1933, Borsuk \cite{B33} made the following conjecture.

\begin{conj}[Borsuk]
Every set of diameter $d > 0$ in the Euclidean $n$-space $\Re^n$ is the
union of $n+1$ sets of diameters less than $d$.
\end{conj}

From the 1930s, this conjecture has attracted a wide interest among
geometers. The frequent attempts to prove it led to results in a number of
special cases: for sets in $\Re^2$ and $\Re^3$, for smooth bodies or sets
with certain symmetries, etc., but the conjecture in general remained open
till 1993, when it was disproved by Kahn and Kalai \cite{K93}. Their
result did not mean that research on this problem stopped: the investigation
of the so-called \emph{Borsuk number} of a bounded set; that is, the minimum
number of pieces of smaller diameters that it can be partitioned into, is
still one of the fundamental problems of discrete geometry.

Since 1933, a large number of generalizations of Borsuk's problem has been
introduced. Without completeness, we list only a few. The \emph{generalized
Borsuk problem} asks to find, for a fixed value of $0 < r < 1$, the minimum
number $m$ such that any set of diameter one in $\Re^n$ can be partitioned
into $m$ pieces of diameters at most $r$ (cf., for example, \cite{G61}).
The \emph{cylindrical Borsuk problem} makes restrictions on the method of
partition (cf. \cite{H05}). Clearly, the
original problem is meaningful for sets in any metric space, for instance,
for finite dimensional normed spaces (cf. \cite{BMS97}) or for binary codes equipped with
Hamming distance. The latter one is called the \emph{$(0,1)$-Borsuk problem}%
, and is investigated, for example, in \cite{R99}, \cite{Z01} and \cite{R02}.
For more information on this problem and its generalizations, the reader is referred
to the survey \cite{R08}.

Our aim is to add another generalization to this list. Our main definition
is the following.

\begin{defn}
\label{defn:multipleBorsuk} Let $S \subset \Re^n$ be a set of diameter $d >
0 $. The smallest positive integer $m$ such that there is a $k$-fold
cover of $S$, with $m$ sets of diameters strictly less than $d$, is
called the \emph{$k$-fold Borsuk number} of $S$. We denote this number by $%
a_k(S)$.
\end{defn}

Recall that a $k$-fold cover of a set $S$ is a family of sets with the property that
any point of $S$ belongs to at least $k$ members of the family.
In this definition, we permit some members of the family to coincide.
We denote the Borsuk number of a set $S$ by $a(S)$. Clearly, $a_1(S) = a(S)$.

Note that Definition~\ref{defn:multipleBorsuk} can be naturally adapted to
almost any variant of the original Borsuk problem, and thus, raises many
open questions that are not examined in this paper. Our goal is to
investigate the properties of the $k$-fold Borsuk numbers of sets in $\Re^n$.

We start with three observations. Then in Section~\ref{sec:planar} we
characterize the $k$-fold Borsuk numbers of planar sets.
In Section~\ref{sec:smooth} we give estimates on the $k$-fold Borsuk numbers of smooth
bodies, centrally symmetric sets, and convex bodies of constant width, and
determine them for Euclidean balls.
In Sections~\ref{sec:3D} and \ref{sec:3Dsym} we examine the $k$-fold
Borsuk numbers of finite point sets in $\Re^3$.
In particular, in Section~\ref{sec:3D} we examine the sets with large $k$-fold Borsuk numbers,
and in Section~\ref{sec:3Dsym} we focus on sets with a nontrivial symmetry group.
Finally, in Section~\ref{sec:remarks} we make an additional remark and
raise a related open question.

During the investigation, $\mathbf{B}^n$ denotes the closed Euclidean unit
ball centered at the origin $o$, and $\S ^{n-1}= \bd \mathbf{B}^n$.

Our first observations are as follows.

\begin{rem}
\label{rem:subadditive} The sequence $a_k(S)$ is subadditive for every $S$.
More precisely, for any positive integers $k,l$, we have $a_{k+l}(S) \leq
a_k(S) + a_l(S)$.
\end{rem}

\begin{rem}
\label{rem:aSis2} For every set $S \subset \Re^n$ of diameter $d > 0$ and
for every $k\geq 1$, we have $a_k(S) \geq 2k$. Furthermore, for every value
of $k$, if $a(S) = 2$, then $a_k(S) = 2k$, and if $a(S) > 2$, then $a_k(S) >
2k$.
\end{rem}

\begin{proof}
Without loss of generality, we may assume that $S$ is compact. Let $[p,q]$
be a diameter of $S$. Since no set of diameter less than $d$ contains both $%
p $ and $q$, any $k$-fold cover of $S$ with sets of smaller diameters
has at least $2k$ elements. Furthermore, if $a(S) = 2$, then
by Remark~\ref{rem:subadditive}, $a_k(S) \leq 2k$ for every value of $k$.

Now assume that $a_k(S)=2k$ for some value of $k$. Let $A_1,A_2, \ldots,
A_{2k}$ be compact sets of diameters less than $d$ that cover $S$ $k$-fold.
or every $k$-element subset $J$ of $I=\{ 1,2, \ldots, 2k \}$, let $\bar{J}%
=I \setminus J$, and let $S_J$ be the (compact) set of points in $S$ that
are covered by $A_i$ for every $i \in J$. Clearly, the union of the sets $%
S_J $ is $S$, when $J$ runs over the $k$-element subsets of $I$. We define
two sets $A$ and $B$ in the following way: for any pair $J$ and $\bar{J}$,
we choose either $S_J$ or $S_{\bar{J}}$ to add to $A$, and we add the other
one to $B$. Then $S \subseteq A \cup B$.

Note that for any pair of points $p,q \in A$, there is an index $i$ such
that $p,q \in A_i$, and thus, $|p-q| \leq \diam A_i < d$. This yields that $%
\diam A < d$. We may obtain similarly that $\diam B < d$, which implies that 
$a(S) = 2$.
\end{proof}

\begin{rem}
\label{rem:boundary} Let $S \subset \Re^n$ be a set of positive diameter.
Then for every value of $k$, $a_k(S) = a_k(\bd S)$.
\end{rem}

\begin{proof}
Without loss of generality, we may assume that $S$ is compact and that $%
\diam S = 1$. Then, clearly, $a_k(S) \geq a_k(\bd S)$ for every $k$.

On the other hand, assume that some sets $Q_{1},Q_{2},\ldots ,Q_{m}$
form a $k$-fold cover of $\bd S$ where each $Q_{i}$ is of
diameter less than one. Without loss of generality, we may assume that $%
Q_{i}\subset S$ for every $i$. Let $\varepsilon >0$ be chosen in such a way
that $\diam Q_{i}<1-2\varepsilon $ for all values of $i$. Then the sets $%
\bar{Q}_{i}=(Q_{i}+\varepsilon \mathbf{B}^{n})\cap S$ form a $k$-fold cover
of $(\bd S+\varepsilon \mathbf{B}^{n})\cap S$ such that $\diam\bar{Q}_{i}<1$%
. Let $T$ denote the set $S\setminus (\bd S+\varepsilon \mathbf{B}^{n})$.
Observe that for any point $p\in T$ and $q\in S$, we have $|p-q|\leq
1-\varepsilon $. Thus, setting $Q_{i}^{\prime }=\bar{Q}_{i}\cup T$ for every 
$i$, we have $\diam Q_{i}^{\prime }\leq \max \{\diam\bar{Q}%
_{i},1-\varepsilon \}<1$, and the sets $Q_{1}^{\prime },Q_{2}^{\prime
}\ldots ,Q_{m}^{\prime }$ form a $k$-fold cover of $S$.
\end{proof}

By Remark~\ref{rem:boundary}, we may imagine the $k$-fold Borsuk number of a convex body $C$
as a painting of the surface of $C$, with $a_k(C)$ colors, such that the diameter of each patch
is less than $\diam C$, and any point on the surface is covered by at least $k$ layers.

Before starting our investigation, we recall two notions from graph theory which we are going to use in the proofs.
First, if $G$ is a graph, then the \emph{$k$-fold chromatic number} $\chi_k(G)$ of $G$ is the smallest integer $m$ with
the property that a $k$-element subset of $\{ 1,2, \ldots m \}$ (called \emph{colors})
can be assigned to each vertex of $G$ in such a way that if two vertices are connected by an edge, then the corresponding subsets are disjoint.
The second notion is the of the \emph{independence number} $\alpha(G)$ of a graph $G$: This number is the cardinality
of the largest subset of the vertex set $V(G)$ of $G$ in which no two edges are connected by an edge.
By the Pigeon-Hole Principle, we clearly have the following inequality.

\begin{rem}
\label{rem:independence}
For any graph $G$, we have
\[
\chi_k(G) \geq \frac{kV(G)}{\alpha(G)}.
\]
\end{rem}

\section{Sets in the Euclidean plane}
\label{sec:planar}

To formulate our main results, we first recall the well-known fact that for
any $S \subset \Re^n$ of diameter $d$, there is a convex body $K \subset
\Re^n$ of constant width $d$ such that $S \subseteq K$. The following
characterization of the Borsuk numbers of plane sets was given by
Boltyanskii (cf.\ \cite{B60}, or alternatively \cite{BG71}).

\begin{thm}[Boltyanskii]
Let $S \subset \Re^2$ be of diameter $d > 0$. The Borsuk number of $S$ is
three if, and only if, there is a unique convex body of constant width $d$,
containing $S$.
\end{thm}

Now we prove the following.

\begin{thm}\label{thm:extension} Let
$S \subset \Re^2$ be a set of diameter $d>0$ with $%
a(S) = 3$, and let $C$ be the unique plane convex body of constant width $d$
that contains $S$. Then for every value of $k$, we have $a_k(S) = a_k(C)$.
\end{thm}

Our proof is based on the following lemma, used by Boltyanskii (cf. for
example, Lemma 8, p. 29, \cite{BG71})

\begin{lem}[Boltyanskii]
\label{lem:Boltyanskii}

For any point $u \in (\bd C) \setminus S$, there is an open
circle arc of radius $d$ in $\bd C$, that contains $u$, such that the center $p$ of the
circle is contained in $\bd C$.
\end{lem}

\begin{proof}[Proof of Theorem~\protect\ref{thm:extension}]
Without loss of generality, let $S$ be compact, and $d=1$. Clearly, $a_k(C)
\geq a_k(S)$. Hence, by Remark~\ref{rem:boundary}, it suffices to show that $%
a_k(\bd C) \leq a_k(S)$.

Assume that $Q_1, Q_2, \ldots, Q_m \subset S$ are sets of diameters less
than one that form a $k$-fold cover of $S$. Without loss of generality, we
may assume that $Q_i \subset C$ for every value of $i$.

Let $\delta>0$ be chosen in a way that $\diam Q_i + 2\delta < 1$ for every
value of $i$. Let $A_j$, where $j=1,2,\ldots, t$ denote the connected
components of $(\bd C) \setminus S$ longer than $2\delta$, and let $q_j$ and 
$r_j$ be the two endpoints of $A_j$. Clearly, there are finitely many such
arcs. First, we note that the sets $\bar{Q}_i=(Q_i + \delta \mathbf{B}^2)
\cap C$ form a $k$-fold cover of $(\bd C) \setminus \left( \sum_{j=1}^t A_j
\right)$, and that the diameter of any of these sets is less than one. We
extend the sets $\bar{Q}_1, \bar{Q}_2, \ldots, \bar{Q}_m$ to cover $k$-fold
all the $A_j$s.

Using Lemma~\ref{lem:Boltyanskii}, for every value of $j$, $A_j$ is an open
unit circle arc with its center $p_j \in \bd C$. Since $C$ is contained in
the intersection of the two unit disks $q_j + \mathbf{B}^2$ and $r_j + 
\mathbf{B}^2$, we obtain that $p_j$ is not a smooth point of $\bd C$, which
yields, by the same lemma, that $p_j \in S$.

Let $A_j^q$ be the set of the points of $A_j$ that are not farther from $q_j$
than from $r_j$. We define $A_j^r$ analogously. Note that for any $u \in A_j$%
, the only point of $C$ at distance one from $u$ is $p_j$. Hence, for any
index $i$ such that $q_j \in \bar{Q}_i$, $p_j$ has a neighborhood disjoint
from $\bar{Q}_i$. This yields that $\diam (\bar{Q}_i \cup A_j^q) < 1$. We
may obtain similarly that if $r_j \in \bar{Q}_i$, then $\diam (\bar{Q}_i
\cup A_j^r) < 1$. Thus, we set 
\begin{equation*}
C_i = \conv\left( \bar{Q}_i \cup \bigcup_{q_j \in \bar{Q}_i} A_j^q \cup
\bigcup_{r_j \in \bar{Q}_i} A_j^r \right),
\end{equation*}
and observe that, by induction on $j$, $\diam C_i < 1$ for every value of $i$, and
that the sets $C_i$ form a $k$-fold cover of $\bd C$.
\end{proof}

Now let us recall the notion of \emph{Reuleaux polygons}. These polygons are
constant width plane convex bodies bounded by finitely many circle arcs of
the same diameter (of radius equal to the width of the body), called the 
\emph{sides} of the polygon. It is a well-known fact that any such polygon
has an odd number of sides (cf.\ \cite{BF48} or \cite{F60}).

\begin{thm}
\label{thm:classification} Let $C$ be a constant width convex body in $\Re^2$%
, and $k$ be a positive integer. If $C$ is a Reuleaux-polygon with $2s+1$
sides, then $a_k(C) = 2k + \left\lceil \frac{k}{s} \right\rceil$, and
otherwise $a_k(C) = 2k+1$.
\end{thm}

\begin{proof}
For simplicity, assume that $\diam C = 1$.

First, consider the case that $C$ is a Reuleaux polygon with $2s+1$ sides.
Let us call the common point of two consecutive sides of $C$ a \emph{vertex}
of $C$. Let these vertices be $p_1, p_2, \ldots, p_{2s+1}=p_0$ in counterclockwise
order in $\bd C$. Consider the diameter graph of the vertex set of $C$: The
vertices of this graph are the vertices of $C$, and two vertices are
connected with an edge if, and only if, they are endpoints of a diameter.
Clearly, this graph is the cycle $C_{2s+1}$, of length $(2s+1)$. By \cite%
{S76} (see also \cite{S95}), the $k$-fold chromatic number of $C_{2s+1}$ is $%
m=2k + \left\lceil \frac{k}{s} \right\rceil$. This implies that $a_k(C) \geq
m$. We note that the inequality $m \geq 2k + \left\lceil \frac{k}{s} \right\rceil$
also follows from Remark~\ref{rem:independence}.

Now we show the existence of some sets $Q_1, Q_2, \ldots, Q_m$, with
diameters less than one, that form a $k$-fold cover of $\bd C$. This, by
Remark~\ref{rem:boundary}, yields the assertion for Reuleaux-polygons. Let $%
A_1, A_2, \ldots, A_m$ be sets of diameters less than one that form a $k$%
-fold cover of the vertices of $C$. Let $\widehat{p_ip_{i+1}}$ denote the
side of $C$ connecting $p_i$ and $p_{i+1}$, and let $G_i$ be the union of
the points of the two sides $\widehat{p_ip_{i-1}}$ and $\widehat{p_ip_{i+1}}$
that are not farther from $p_i$ than from $p_{i-1}$ and $p_{i+1}$,
respectively. Observe that the sets $Q_j = \bigcup_{p_i \in A_j} G_j$ form a 
$k$-fold cover of $\bd C$, and that their diameters are strictly less than
one, which readily implies the assertion.

In the remaining part we assume that $C$ is \emph{not} a Reuleaux-polygon.
By Remark~\ref{rem:aSis2}, we have that $a_k(C) \geq 2k+1$. Hence, it
suffices to construct a family of $2k+1$ sets of diameters less than one
that form a $k$-fold cover of $C$ or, by Remark~\ref{rem:boundary}, $\bd C$.

First, we carry out this construction for $C=\frac{1}{2}\mathbf{B}^{2}$.
Consider $2k+1$ distinct diameters of $C$. Let these be $%
[p_{1},q_{1}],[p_{2},q_{2}],\ldots ,[p_{2k+1},q_{2k+1}]$, where the notation
is chosen in such a way that the points $p_{1},p_{2},\ldots ,p_{2k+1}=p_{0}$
and $q_{1},q_{2},\ldots ,q_{2k+1}=q_{0}$ are in counterclockwise order in $\bd C$%
, and for every $i$, the shorter arc connecting $p_{i}$ and $p_{i+1}$
contains exactly one of the $q_{j}$s (by exclusion, this point is $q_{i+k}$,
cf. Figure~\ref{fig:circle}).
Observe that the points $p_{i}$ have the property that the diameter
containing any one of them divides $\bd C$ into two open half circles, each
of which contains exactly $k$ of the remaining $2k$ points. Let $A_{i}$
denote the shorter arc in $\bd C$ connecting $p_{i}$ and $p_{i+k}$. Observe
that these sets form a $k$-fold cover of $\bd C$, and that their diameters
are less than one.

\begin{figure}[here]
\includegraphics[width=0.4\textwidth]{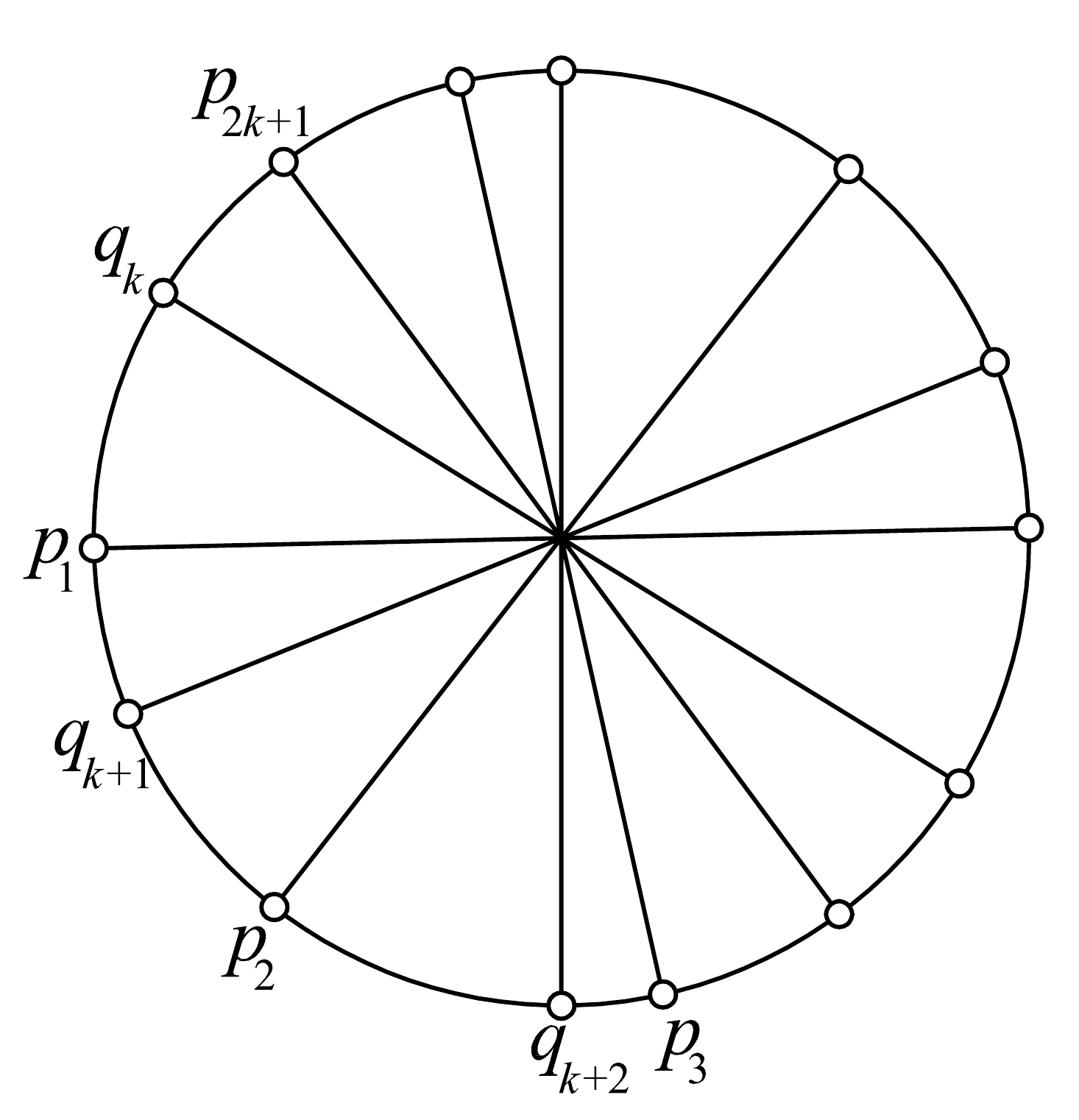}
\caption[]{Covering a Euclidean disk}
\label{fig:circle}
\end{figure}

In the last step, we show that a similar family can be constructed for any $%
C $ that is not a Reuleaux-polygon. Before we do that, we recall the
following simple property of plane convex bodies of constant width:

\begin{itemize}
\item No two diameters of a plane convex body of constant width are disjoint.
\end{itemize}

For any $p \in \bd C$, let $G(p) \subset \S ^1$ be the \emph{Gaussian image}
of $p$; that is, the set of the external unit normal vectors of the lines
supporting $C$ at $p$. Observe that if $G(p)$ is not a singleton, then it is
a closed arc in $\S ^1$. Furthermore, in this case $p$ is not a smooth point
of $\bd C$, and thus, by Lemma~\ref{lem:Boltyanskii} in the previous proof,
the locus of the other endpoints of the diameters starting at $p$ is a
closed unit circle arc in $\bd C$. Apart from the endpoints, the points of
this arc are smooth points of $\bd C$, and thus, their Gaussian images are
singletons. This yields that if $G(p)$ is not a singleton, then, apart from
its endpoints, the points of $-G(p)$ are decomposed into singleton Gaussian
images.

Since $C$ is not a Reuleaux-polygon, we may choose $2k+1$ diameters of $%
\mathbf{B}^2$, say $[p_1,q_1], [p_2,q_2], \ldots, [p_{2k+1},q_{2k+1}]$, such
that the Gaussian image of any point of $\bd C$ intersects at most one of
them. Indeed, as $\S ^1$ is not covered by the union of finitely many
Gaussian images and their antipodal arcs, at least one of the following
holds:

\emph{Case 1}, There are at least $2k+1$ Gaussian images in $\S ^1$ that are
not singletons: then we may choose $2k+1$ such arcs, and pick one point
from each, different from the endpoints of the arc, as an endpoint of one of
the chosen diameters.

\emph{Case 2}, There are less than $2k+1$ Gaussian images that are not
singletons. In this case there is an open arc $I \subset \S ^1$ such that
both $I$ and $-I$ are decomposed into singleton Gaussian images. Thus, we
may choose $2k+1$ pairwise distinct diameters with all their endpoints in $I
\cup (-I)$.

Let us label the endpoints of these diameters as in the case that $C$ is a
Euclidean disk. That is, assume that the points $p_1, p_2, \ldots,
p_{2k+1}=p_0$ are in counterclockwise order in $\bd C$, and for every $i$, the
(counterclockwise) directed arc connecting $p_i$ and $p_{i+1}$ contains exactly
$q_{i+k}$ from amongst the $q_j$s.
Let $G^{-1}(u)$ denote the (unique) point $v$ of $\bd C$
with the property that $u \in G(v)$. For every $i$, let $F_i$ denote the
closed arc of $\bd C$, with endpoints $G^{-1}(p_i)$ and $G^{-1}(p_{i+k})$,
and containing $G^{-1}(p_{i+j})$ for $j=1,2,\ldots,k-1$. Observe that since
any two diameters of $C$ have a nonempty intersection, no arc $F_i$ contains
the endpoints of a diameter, and thus, $\diam F_i < 1$. On the other hand,
these arcs form a $k$-fold cover of $\bd C$, which implies that $a_k(C) =
2k+1$.
\end{proof}

\section{Centrally symmetric sets and smooth bodies}

\label{sec:smooth}

Two of the special cases for which Borsuk's original conjecture is proven,
are when the set is centrally symmetric, or is a smooth convex body
(cf. \cite{R71}, \cite{H45} and \cite{H46}). The proofs in both cases are
based on reducing the problem to the Euclidean $n$-ball $\mathbf{B}^n$, and
then to finding $a(\mathbf{B}^n)$. In this section, we investigate the $k$%
-fold Borsuk numbers of these sets in the same way.

For preciseness, we first remark that we call a set $S \subset \Re^n$ a \emph{%
smooth body}, if $S$ is homeomorphic to $\mathbf{B}^n$, and its boundary is
a $C^1$-class submanifold of $\Re^n$.

\begin{thm}
\label{thm:smooth} Let $S \subset \Re^n$ be a set of diameter $d > 0$.

\begin{itemize}
\item[(1)] If $S$ is a smooth body or centrally symmetric, then for every $k$%
, we have $a_k(S) \leq a_k(\mathbf{B}^n)$.

\item[(2)] If $S$ is a convex body of constant width, then for every $k$, we
have $a_k(S) \geq a_k(\mathbf{B}^n)$.

\item[(3)] For every $k$, we have $a_k(\mathbf{B}^n) = 2k+n-1$.
\end{itemize}
\end{thm}

Clearly, this theorem implies that if $S$ is a smooth convex body of
constant width, then for every $k$, $a_k(S) = 2k+n-1$.

\begin{proof}
Let $d = 1$.

First we examine the case that $S$ is a smooth body. For every $p \in \bd S$%
, let $G(p)$ denote the Gaussian image of $p$; that is, the unique external
unit normal vector of $S$ at $p$. Then $G : \bd S \to \S ^{n-1}$ is a
continuous mapping. Observe that if $[p,q]$ is a diameter of $S$, then $p,q
\in \bd S$, and $G(p)=-G(q)=p-q$. Thus, any $k$-fold cover of $\S ^{n-1}$ by 
$m$ sets of smaller diameters induces a $k$-fold cover of $\bd S$, and thus $%
S$, by $m$ sets of smaller diameters. This shows that $a_k(S) \leq a_k(%
\mathbf{B}^n)$.

Now, assume that $S$ is a (not necessarily smooth) convex body of constant
width. Then every point of $\bd S$ is an endpoint of some diameter, and
thus, a $k$-fold cover of $\bd S$ induces a $k$-fold cover of $\S ^{n-1}$
like in the previous paragraph. This yields $a_k(S) \geq a_k(\mathbf{B}^n)$
for every $k$.

Next, let $S$ be symmetric to the origin. Observe that $S \subseteq \frac{1}{%
2} \mathbf{B}^n$. Then the set $S_d = S \cap \left( \frac{1}{2} \S ^{n-1}
\right)$ contains all the points of $S$ that are endpoints of some diameter,
and thus, the inequality $a_k(S) \leq a_k(\mathbf{B}^n)$ follows by an
argument similar to the one used for smooth bodies.

We are left to show that $a_k(\mathbf{B}^n) = 2k+n-1$, or equivalently, that 
$a_k(\S ^{n-1}) = 2k+n-1$. To show that $a_k(\S ^{n-1}) \geq 2k+n-1$, we
follow the idea of the proof for the usual Borsuk number of $\S ^{n-1}$.

Consider a $k$-fold cover $\mathcal{F} = \{ Q_1, Q_2, \ldots, Q_m \}$ of $\S %
^{n-1}$, with closed sets, such that no element of $\mathcal{F}$ contains a
pair of antipodal points. Let us define the function $f : \S ^{n-1} \to
\Re^{n-1}$ as 
\begin{equation*}
f(x) = (\dist(x,A_1), \dist(x,A_1), \ldots, \dist(x,A_{n-1}) )
\end{equation*}
This function is clearly continuous, and hence, by the Borsuk-Ulam Theorem,
there is a point $p \in \S ^{n-1}$ such that $f(p) = f(-p)$. If some
coordinate of $f(p)$ is zero, then both $p$ and $-p$ are elements of one of the
sets $A_1, A_2, \ldots, A_{n-1}$; a contradiction. Thus, $f(p)$ has no
coordinate equal to zero, which means that neither $p$ nor $-p$ belongs to $A_1 \cup
\ldots A_{n-1}$. Since $p$ and $-p$ are antipodal points, there are at least 
$2k$ elements of $\mathcal{F}$ that contain one of them, which yields that $%
m \geq 2k+n-1$. On the other hand, Gale proved (cf. Theorem II' in \cite%
{Ga56}) the existence of a family of $2k+n-1$ open hemispheres of $\S ^{n-1}$
that form a $k$-fold cover of $\S ^{n-1}$. Since contracting these open
hemispheres one by one yields a $k$-fold cover of $\S ^{n-1}$ with $2k+n-1$
closed spherical caps of radii strictly less than $\frac{\pi}{2}$, we obtain
that $a_k(\mathbf{B}^n) = 2k+n-1$.
\end{proof}

\section{Multiple Borsuk numbers of finite point sets in Euclidean $3$-space}\label{sec:3D}

The fact that the (usual) Borsuk numbers of finite sets in $3$-space are at
most four was first shown by Heppes and R\'ev\'esz in \cite{HR56}, and it
also follows from the proof of V\'azsonyi's conjecture (cf. \cite{H57}, \cite%
{G56}, \cite{KMP10} or \cite{S08}), that stated that in any set $S \subset
\Re^3$ of cardinality $m$, $\diam S$ is attained between at most $2m-2$
pairs of points; or in other words, that the diameter graph of any set of $m$
points in $\Re^3$ has at most $2m-2$ edges. Later we use some of the ideas of these proofs.

A complete characterization of the Borsuk numbers of finite sets in $\Re^3$, even of those
with $a(S)=4$ looks hopeless: indeed, by (2) of Theorem~\ref{thm:smooth},
if $S$ is the vertex set of a Reuleaux-polytope in $\Re^3$, then $a(S)=4$,
and a result of Sallee \cite{S70} yields that the family of
Reuleaux-polytopes in $\Re^3$ is an everywhere dense subfamily of the family
of convex bodies of constant width in $\Re^3$.
Thus, unlike in Section~\ref{sec:planar}, in this and the next sections
we restrict our investigation to point sets with some special properties.

The main goal of this section is to find the finite sets $S \subset \Re^3$
with $a_k(S) = 4k$ for every value of $k$.
We observe that for a finite set $S \subset \Re^3$, Remark~\ref{rem:subadditive}
readily implies that $a_k(S) \leq 4k$ for every value of $k$, where we have equality, for example,
for regular tetrahedra.

Using the next, naturally arising concept, we may rephrase our question in a different form.

\begin{defn}\label{defn:fractionalBorsuk}
Let $S \subset \Re^n$ be of diameter $d > 0$. Then the quantity
\[
a_{frac}(S) = \inf \left\{ \frac{a_k(S)}{k} : k=1,2,3\ldots \right\}
\]
is called the \emph{fractional Borsuk number} of $S$.
\end{defn}

Clearly, $a_{frac}(S) \leq a(S)$ for every set $S$.

\begin{prob}\label{prob:fractional}
Prove or disprove that if $S \subset \Re^3$ is a finite point set with
$a_{frac}(S) = 4$, then its diameter graph contains $K_4$ as a subgraph.
\end{prob}

We give only a partial answer to this problem.
During the investigation, we denote the diameter graph of $S$ by $G_S$, and recall
that the \emph{girth} of a graph $G$ is the length of a shortest cycle in $G$.
We denote this quantity by $g(G)$, and note that
for any finite set $S \subset \Re^3$, we have $a_k(S) = \chi_k(G_S)$,
where $\chi_k(G_S)$ denotes the $k$-fold chromatic number of $G_S$.

Our main result is the following.

\begin{thm}\label{thm:finite4k}
Let $S \subset \Re^3$ be a finite set with $g(G_S) > 3$.
Then $a_k(S) < 4k$ for some value of $k$.
\end{thm}

This theorem may be rephrased in the following form:
For any finite set $S \subset \Re^3$ with $a_{frac}(S)=4$, $G_S$ contains $K_3$ as a subgraph. 
The proof is based on Lemma~\ref{lem:finite_oddcoloring}.

\begin{lem}\label{lem:finite_oddcoloring}
There is an $m$-fold $(2m+1)$-coloring of the $(2m+1)$-cycle $C$ with
the property that any two nonconsecutive vertices have a common color.
\end{lem}

\begin{proof}[Proof of Lemma~\ref{lem:finite_oddcoloring}]
Let the vertices of $C$ be $v_1,v_2,\ldots,v_{2m+1}=v_0$ in counterclockwise order.
Let the colors be $1,2, \ldots, 2m+1$.
We define a $3$-coloring of $C$ as follows with the colors $t,t+1,2m+1$,
where $t \in \{1,3,\ldots,2m-1 \}$: only $v_t$ is colored with $2m+1$, and the
vertices $v_{t+1}, v_{t+2}, \ldots$ are colored with $t$ and $t+1$, alternately.
It is easy to see that the union of these $m$ $3$-colorings is an $m$-fold
$(2m+1)$-coloring of $C$.

Consider any two vertices $v_i$ and $v_j$ with $|i-j| \geq 2$.
Since $C$ is an odd cycle, exactly one of the two connected components of $C \setminus \{ v_i, v_j\}$
contains an even number of vertices.
If this component does not contain a vertex colored with the color $2m+1$,
then $v_i$ and $v_j$ are $v_{2m-1}$ and $v_1$, both of which are colored with $2m+1$.
If the even component contains the vertex $a_t$ colored with $2m+1$,
then both $v_i$ and $v_j$ are colored either with $t$ or with $t+1$.
\end{proof}

\begin{proof}[Proof of Theorem~\ref{thm:finite4k}]
Let $C$ be a shortest odd cycle in $G_S$, of length $2m+1 \geq 5$.
We show that $\chi_m(G_S) \leq 4m-1$.
Note that if $G_S$ contains no odd cycle, then it is bipartite, and thus, the statement follows from $\chi(G_S) = 2$.

By \cite{D00}, any odd cycle of $G_S$ intersects $C$, or in other words, $G_S \setminus C$ is a bipartite graph.
Let the two parts of $V(G_S \setminus C)$ in this partition be $V_1$ and $V_2$.
Clearly, since $G_S$ contains no triangle, no vertex of $V_1$ is connected to two consecutive
vertices of $C$.
Furthermore, no vertex of $V_1$ is connected to more than two vertices of $C$.
Indeed, if a vertex $v$ is connected to the distinct vertices $v_1, v_2, v_3 \in C$, then there is a path in $C$,
of odd length at most $2m-3$, that connects two of $v_1, v_2$ and $v_3$.
This yields that $C$ is not a shortest odd cycle of $G_S$; a contradiction (for this argument, cf. also \cite{D00}).

Now we define an $m$-fold $(4m-1)$-coloring of $G_S$.
We color each vertex of $V_2$ with the colors $3k,3k+1, \ldots,4k-1$,
and use only the remaining colors for $C \cup V_1$.
We color $C$ in the way described in Lemma~\ref{lem:finite_oddcoloring}, using only the colors $1,2,\ldots,2m+1$.
We color the vertices of $V_1$, using $1,2,\ldots,3m-1$, in the following way.
Consider a vertex $v \in V_1$. Then $v$ is connected to at most two vertices of $C$,
which are not consecutive. Hence, by Lemma~\ref{lem:finite_oddcoloring},
there are at most $2m-1$ colors used for coloring them.
Thus, there are at least $(3m-1)-(2m-1)=m$ colors, from amongst $1,2,\ldots,3m-1$,
that do not color any neighbor of $v$ in $C$.
We color $v$ with $m$ such colors.
\end{proof} 

In the remaining part we show that the statement of Problem~\ref{prob:fractional}
holds for any set $S$ with $\card S \leq 7$.
We start with finding the $4$-critical subsets of
diameter graphs of sets in $\Re^3$ of at most seven points.
Recall that a graph $G$ is $m$-critical if $\chi(G)=m$, and for any proper subgraph $H$ of $G$,
$\chi(H) < m$.

\begin{lem}\label{thm:characterization}
If $S \subset \Re^3$ with $\card S \leq 7$ and $a(S)=4$, and $H$ is a $4$-critical
subgraph of $G_S$, then $H$ is either $K_4$, or the wheel graph $W_6$,
or the Mycielskian $\mu(C_3)$ of the $3$-cycle $C_3$ (cf. Figure~\ref{fig:4critical}).
\end{lem}

\begin{figure}[here]
\includegraphics[width=\textwidth]{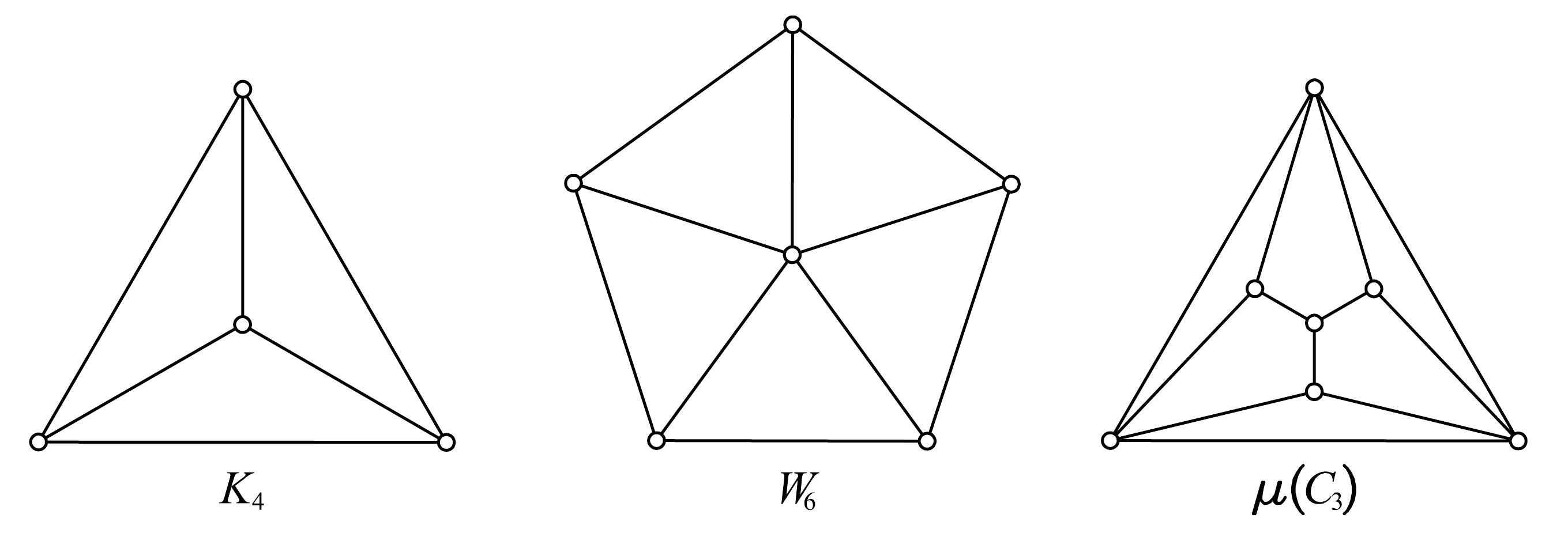}
\caption[]{$4$-critical subgraphs of diameter graphs}
\label{fig:4critical}
\end{figure}

\begin{proof}
By the proof of V\'azsonyi's conjecture, $G_S$ has at most $2 \card S-2$ edges, and by \cite{D00},
any two odd cycles of $G_S$ intersect. Clearly, these properties hold also for all the subgraphs of $G_S$.
Thus, it suffices to prove the following, slightly more general statement:
If $G$ is a $4$-critical graph with at most $m \leq 7$ vertices and at most $2m-2$ edges such that
any two odd cycles of $G$ intersect, then $G$ is either $K_4$, or $W_6$ or $\mu(C_3)$.

Answering a question of Toft \cite{T75} it was
proven in \cite{J95} that if a $4$-critical graph has at
least one vertex of degree $3$, then the graph contains a \emph{fully odd subdivision}
of $K_{4}$ as a subgraph, where a fully odd subdivision of a graph $H$ is a graph, obtained from $H$
in a way that the edges of $H$ are replaced by paths with odd numbers of edges.

Since our graph $G$ has at most $2m-2$ edges and, being $4$-critical, the degree of any vertex is at least $3$,
$G$ has at least four vertices of degree $3$, and hence, by \cite{J95}, it contains a fully odd
subdivision of $K_4$.
If $G$ is not $K_4$, then, as $G$ is $4$-critical,
this subdivision does not coincide with $K_4$, and thus, $m \geq 6$.

We leave it to the reader to show that the only $4$-critical graph with six vertices
and satisfying our conditions is $W_6$.
We deal only with the case $m=7$. For the proof we use the notations in Figure~\ref{fig:subdivision}.
Since $G$ has at most $12$ edges, there are at most four edges not shown in
Figure~\ref{fig:subdivision}.
Note that as $G$ contains no disjoint triangles and the degree of every vertex is at least $3$,
$a_7$ is connected to exactly one of $a_5$ or $a_6$. By symmetry, we may assume that
$a_5a_7$ is an edge, and $a_6a_7$ is not.
This implies also that the degrees of $a_5$, $a_6$ and $a_7$ are $3$, and that $G$ has exactly $12$ edges.

\begin{figure}[here]
\includegraphics[width=0.4\textwidth]{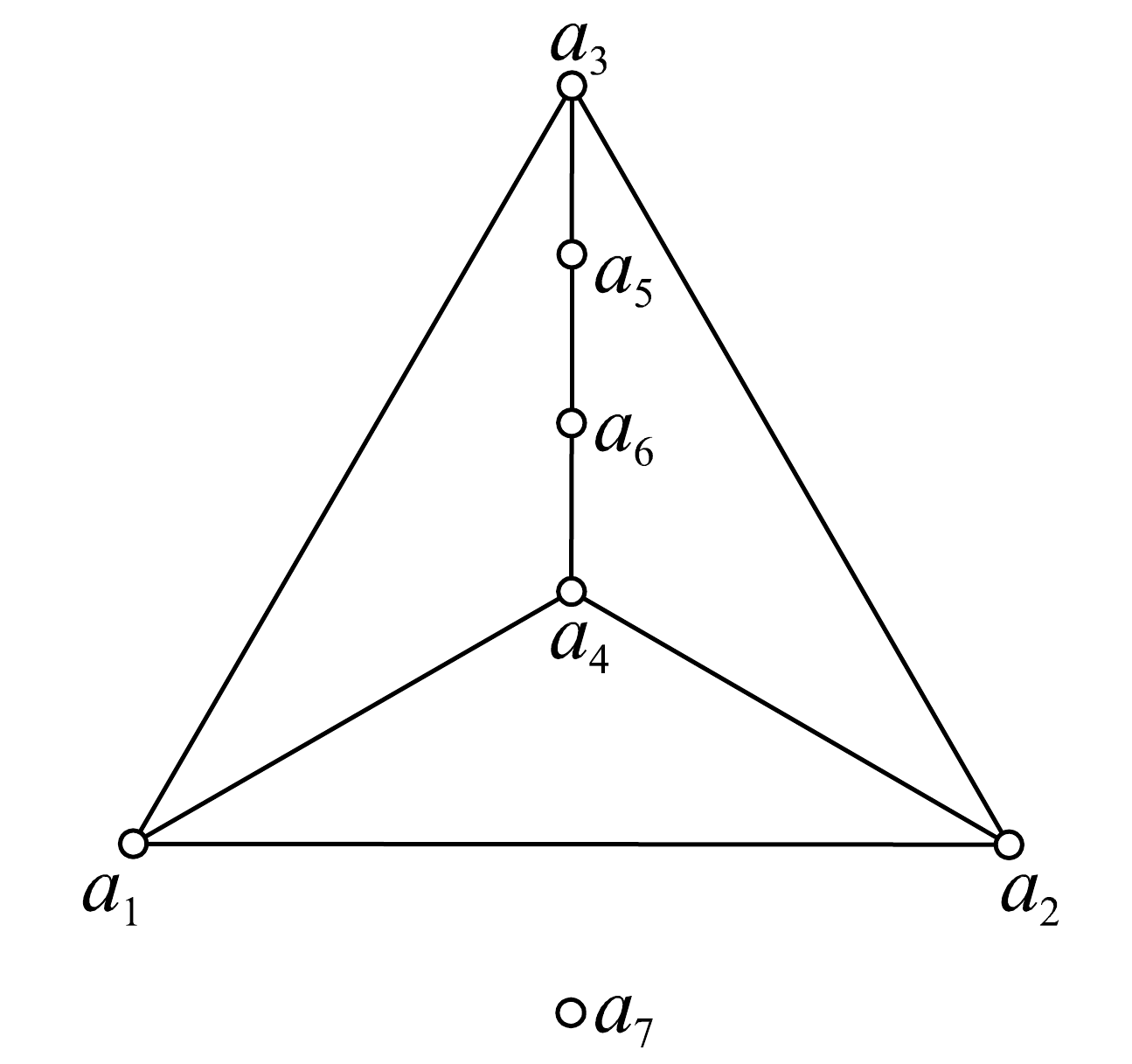}
\caption[]{An illustration for the proof of Lemma~\ref{thm:characterization}}
\label{fig:subdivision}
\end{figure}

By a similar argument, we may obtain that exactly one of $a_1a_6$ and $a_2a_6$ is an edge, say $a_1a_6$.
Then the two additional edges of $G$ connect $a_7$ to two of $a_1,a_2,a_3$ and $a_4$.
It is an elementary exercise to check that if these edges are not $a_2a_7$ and $a_4a_7$, then $G$ is $3$-colorable
or contains disjoint triangles. But if they are $a_2a_7$ and $a_4a_7$, then $G = \mu(C_3)$, which finishes the proof.
\end{proof}

\begin{thm}\label{thm:frac2}
If $S \subset \Re^3$ with $\card S \leq 7$, then either $a_k(S) < 4k$ for every $k > 1$,
or $G_S$ contains $K_4$ as a subgraph.
\end{thm}

\begin{proof}
If $a_1(S) \leq 3$ or $a_2(S) \leq 7$, then the assertion readily follows from Remark~\ref{rem:subadditive}.
Thus, we may assume that $\card S \leq 7$ and that $a_2(S) = 2a_1(S) = 8$,
or equivalently, that $2 \chi(G_S) = \chi_2(G_S) = 8$.
As a consequence of Lemma~\ref{thm:characterization}, $G_S$ contains
$K_4$, $W_6$ or $\mu(C_3)$ as a subgraph.

If $G_S$ contains $K_4$, we are done.
If $G_S$ contains $\mu(C_3)$ as a subgraph, then $G_S = \mu(C_3)$, since
$\mu(C_3)$ has $7$ vertices and $12$ edges, and by V\'azsonyi's problem
$G_S$ has no more than $12$ edges. It is easy to see that $\chi_2(\mu(C_3)) = 7$,
which immediately implies the assertion.

Now we deal with the case that $G_S$ contains $W_6$ as a subgraph.
If $\card S=6$, then by V\'azsonyi's problem $G_S = W_6$, and thus,
$a_2(S) = \chi_2 (W_6) = 7$.
Assume that $\card S = 7$.
Then there are at most two edges of $G_S$ not contained in the subgraph $W_6$.
We may assume that $W_6$ is an induced subgraph of $G_S$, since
if an additional edge of $G_S$ connects two vertices of $W_6$, then
$G_S$ contains $K_4$ as a subgraph.
Thus, $G_S$ is obtained from $W_6$ by adding an additional vertex, and connecting
it to at most two vertices. Depending on the choice of these vertices, it is an elementary
exercise to find a $2$-fold $7$-coloring of $G_S$ in each case.
\end{proof}

\section{The Borsuk numbers of symmetric finite point sets in $\Re^3$}\label{sec:3Dsym}

In this section, our aim is to examine the Borsuk numbers of finite
point sets in $\Re^3$ with their Borsuk numbers equal to four, and with a nontrivial symmetry group.
Our project is motivated by a result of Rogers \cite{Ro71},
who proved Borsuk's conjecture for $n$-dimensional sets with symmetry groups
containing that of a regular $n$-dimensional simplex.

In our investigation, for a finite set $S \subset \Re^3$, we denote the symmetry group of $S$
by $\Sym(S)$, the symmetry group of a regular tetrahedron by ${\mathcal{T}}_4$,
and that of a regular $k$-gon by ${\mathcal{D}}_k$.

\begin{thm}\label{thm:tetrahedron}
Let $S \subset \Re^3$ be a finite set with ${\mathcal{T}}_4 \subseteq \Sym(S)$.
If $g(G_S) = 3$, then $G_S$ contains $K_4$ as a subgraph.
\end{thm}

\begin{proof}
Without loss of generality, let $\diam S = 1$, and let the regular
tetrahedron, with symmetry group ${\mathcal{T}}_4$ and with unit edge
length, be $T$.

Let $a,b,c \subseteq S$ be the vertices of a regular triangle of unit edge
length. Let $M$ be any plane reflection contained in $\Sym (S)$. Then the
points $M(a), M(b), M(c)$ are contained in $S$. By \cite{D00} or \cite{S08},
any two odd cycles of $G_S$ intersect, and thus the sets $\{a,b,c \}$
and $\{ M(a), M(b), M(c) \}$ are not disjoint. If a point is the
reflecion of another one, say $b = M(a)$, then, clearly, $a = M(b)$, and
then $c$ is on the reflection plane; that is, $M(c) = c$. If a point is its
own reflection, then the point is on the reflection plane. Thus, we have
shown that each reflection plane of $\Sym (S)$ contains at least one vertex
from any $3$-cycle in $G_S$.

We leave it to the reader to show that since $\diam S = 1$, then $\{ a,b,c
\} $ does not contain the center of $T$. Thus, there is an axis of rotation in $T_4$
that is disjoint from $\{ a,b,c \}$. Let $R$ be a rotation with angle $\frac{2\pi}{3}$
around this axis. Then the triples $\{ a,b,c \}$ and $\{ R(a), R(b), R(c) \}$
have a point in common; say, $b=R(a)$ (note that no point is the rotated
copy of itself). In this case $\{ a,b,R(b) \}$ is a $3$-cycle in $G_S$ which
is invariant under $R$. As any $3$-cycle has a point on each reflection
plane, it implies that the vertices of this cycle are on the other three axes
of rotation. Applying the symmetries of ${\mathcal{T}}_4$ to these vertices we obtain the
vertices of a regular tetrahedron of unit edge length, which readily implies
the assertion.
\end{proof}

\begin{rem}
Combining Theorems~\ref{thm:tetrahedron} and \ref{thm:finite4k}, we have that
if for some finite set $S \subset \Re^3$ we have ${\mathcal{T}}_4 \subseteq \Sym(S)$,
and $a_k(S) = 4k$ for every $k$, then $G_S$ contains $K_4$ as a subgraph. 
\end{rem}

We note that by \cite{S08}, for every finite set $S \subset \Re^3$, $G_S$  
can be embedded in the projective plane. On the other hand, an example in \cite{S08}
shows that not all these graphs are planar.

\begin{rem}
It is known that the chromatic number of every triangle-free planar graph is at most three.
Thus, Theorem~\ref{thm:tetrahedron} yields that if $G_S$ is planar
and ${\mathcal{T}}_4 \subseteq \Sym(S)$, then $S$ contains $K_4$ as a subgraph.
\end{rem}

\begin{prob}
Prove or disprove that if $S \subset \Re^3$ is a finite set with ${\mathcal{T}}_4 \subseteq \Sym(S)$
and with $a(S) = 4$, then $G_S$ contains $K_4$ as a subgraph.
\end{prob}

Our next aim is to examine sets $S$, with $a(S) = 4$ and with ${\mathcal{D}}%
_{2k+1} \subseteq \Sym S$ for some integer $k \geq 1$. We construct a family
of sets satisfying these conditions.

In the construction we use the notion of the \emph{$p$-Mycielskian} of a graph $G$ (cf. \cite{T01}),
denoted by $\mu_p(G)$.
We regard the \emph{wheel graph $W_{2k+2}$} as the $0$th Mycielskian of the odd cycle $C_{2k+1}$.

\begin{thm}\label{thm:Mycielskian}
For any $p \geq 0$, and $k > 0$, $\mu_p(C_{2k+1})$
is the diameter graph of a finite set $S \subset \Re^3$.
\end{thm}

\begin{proof}
Let $p_{1},p_{2},\ldots ,p_{2k+1}=p_{0}$ be the vertices of a regular $%
(2k+1) $-gon in the $(x,y)$-plane, centered at the origin. Assume that the
diameter of the point set is $r\leq 1$. Consider the points $q=(0,0,\sqrt{%
1-r^{2}})$ and $r=(0,0,\sqrt{1-r^{2}}-1)$. Note that for every $i$, $%
||p_{i}-q||=1$ and $||p_{i}-r||<1$.

Let $v_i$ denote the inner unit normal vector of the supporting plane of
the pyramid $\conv \{ p_1, \ldots, p_{2k+1}, q\}$
passing through the points $p_{i\pm k}$ and $q$. An elementary computation shows
that $\langle v_i, p_j - q \rangle \geq 0$ for every $i$ and $j$,
and if $j \neq i, i+1$, then we have strict inequality. Thus, for every $i$,
we may choose a point $q_i$ such that the points $q_i$ are the vertices of a
regular $(2k+1)$-gon (and have equal $z$-coordinates), $| p_j - q_i | \leq
1$ with equality if and only if $j = i \pm k$. Furthermore, if the points $%
q_i$ are sufficiently close to $q$, then for the point $r^{\prime }$ on the
negative half of the $z$-axis that satisfies $||q_i-r|| = 1$, we have $||
p_i - r^{\prime }|| < 1$.

Now, to obtain the required $p$-Mycielskian, we start with the wheel graph $%
W_{2k+2}$. This can be realized as the vertex set $V_1$ of a pyramid, with a
regular $(2k+1)$-gon of diameter one as its base, and with the property that
the distance of its apex from any other vertex is one. To obtain a $p$%
-Mycielskian, we may apply the procedure described in the first two
paragraphs $(p-1)$ times.
\end{proof}

\begin{rem}\label{rem:Mycielskianproperties}
Let $S$ be a point set with $G_S = \mu_p(C_{2k+1})$. Then $a(S) = 4$ for every $p$ and $k$.
Hence, since for $k \geq 2$ and $p > 1$ $\mu_p(C_{2k+1})$ is triangle-free,
it is not a planar graph.
On the other hand, it is easy to see that the number of edges in $G_S$
is equal to $2 \card S -2$.
Thus, these sets form an infinite family of nonplanar V\'azsonyi-critical graphs.
\end{rem}

\begin{rem}\label{rem:Mycielskiankfold}
Clearly, if $p=0$, then we have
$a_k(S) = k + \chi_k(C_{2m+1}) = 3k+\lceil \frac{k}{m} \rceil$. By
\cite{L09}, for $p=1$, we have 
\begin{equation*}
a_k(S)= \left\{%
\begin{array}{cl}
4 & \hbox{if } k=1, \\ 
\frac{5k}{2}+1 & \hbox{if } k \hbox{ is even}, \\ 
2k + \frac{k+3}{2}, & \hbox{if } k \hbox{ is odd and } k \leq m \leq \frac{%
3k+3}{2}, \hbox{ and} \\ 
2k + \frac{k+5}{2}, & \hbox{if } k \hbox{ is odd and } m \geq \frac{3k+5}{2}.%
\end{array}
\right.
\end{equation*}
\end{rem}

\begin{thm}\label{thm:regularpolygon}
Let $S \subset \Re^3$ be a finite set with ${\mathcal{D}}_{2m+1} \subseteq \Sym(S)$
for some $m \geq 2$. If $a(S) = 4$ and $g(G_S) = 3$, then $G_S$ contains a topological
wheel graph $W_{2m+2}$ as a subgraph.
\end{thm}

In the proof, we use the following lemma.

\begin{lem}\label{lem:Swanepoel}
If $S \subset \Re^3$ is a finite set such that $\Sym%
(S) $ contains a reflection about the plane $H$, then every odd cycle of $%
G_S $ has a vertex on $H$.
\end{lem}

\begin{proof}
Swanepoel (cf. Theorem 2 of \cite{S08}) showed that for any $S\subset \Re^{3}$,
$G_{S}$ has a bipartite double cover, with a centrally symmetric
drawing on $\S ^{2}$: in this drawing, any point $p$ is represented by a
pair of antipodal points $p_{b}$ and $p_{r}=-p_{b}$ that are colored
differently, and a diameter of $S$, connecting $p$ and $q$, corresponds to the two edges
$p_b q_r$ and $p_r q_b$. In his construction, the point $p_r$ representing $p$
is an arbitrary relative interior point of the conic hull of the diameters of $S$ starting at $p$.  
Using the geometric properties of the conic hulls of the diameters,
he concluded that any two odd cycles of $G_{S}$,
which are represented by centrally symmetric closed curves on $\S ^{2}$,
have a common vertex.

Now consider the plane $H^{\prime }$, parallel to $H$ and containing $o$.
We apply the construction of Swanepoel with a special choice of points.
For any $p \in S$, let $CH_p$ denote the conic hull of the diameters of $S$, starting at $p$.
Then we choose $p_r \in \S^2$ as the projection of the center
of gravity of $\B^3 \cap CH_p$ on $\S^2$ from $o$.
Clearly, $p_r$ is on $H^{\prime }$ if, and only if $p$ is on $H$.

Consider an odd cycle $C$ in $G_S$. If its vertex set is symmetric about $H$, then
$H$ contains one of the vertices.
Assume that $C$ is not symmetric about $H$, and let $C'$ denote its reflected copy about $H$.
Clearly, the curves representing $C$ and $C'$ on $\S^2$ are symmetric about
$H^{\prime }$, and thus, they intersect on $H'$.
By Lemmas 1 and 2 of \cite{S08}, these common points belong to common vertices of $C$ and $C'$,
which yields that both cycles have a vertex on $H$.
\end{proof}

\begin{proof}[Proof of Theorem~\protect\ref{thm:regularpolygon}]
Assume that $\diam S = 1$.

By Lemma~\ref{lem:Swanepoel}, any odd cycle, and in particular any triangle,
of $G_S$ contains a point on each plane of symmetry in $\Sym(S)$. Since $\Sym%
(S)$ contains at least $2m+1 \geq 5$ symmetry planes, any triangle $T$ of $%
G_S$ has a vertex on the axis $L$ of the rotations of ${\mathcal{%
D}}_{2m+1}$. Clearly, this triangle $T$ has at most two vertices on $L$.
If $T$ has exactly two vertices on $L$, then the diameter of the union of the rotated copies of $T$
is stricly greater than one; a contradiction. Thus, we have that $T$ has exactly one
vertex on $L$, which we denote by $a$. Let the remanining two vertices of $T$ be $b$ and $c$.
Let $b=b_{1},b_{2},\ldots ,b_{2m+1}$, and $c=c_{1},c_{2},\ldots ,c_{2m+1}$
denote the rotated copies of $b$ and $c$, respectively, about $L$.

First, consider the case that the points $b_{i}$ and $c_{j}$ are pairwise
distinct. Let $C$ be a shortest odd cycle that does not contain $a$. Such a
cycle exists, as otherwise $G\setminus \{a\}$ contains no odd cycle, and $%
\chi (G)=3$. Since any two odd cycles intersect, $C$ contains at least one
point from each pair $\{b_{i},c_{i}\}$.
Thus, the required subgraph is defined as the union of $C$, $a$, and for each 
$i$ an edge connecting $a$ to either $b_{i}$ or $c_{i}$ on $C$.

Finally, assume that from amongst the $b_i$s and the $c_j$s there are
coinciding vertices. Note that since they are not on $L$, we have that $b_i
= c_j$ for some $i$ and $j$. But then $\{ b_1, b_2, \ldots, b_{2m+1} \}= \{
c_1, c_2, \ldots, c_{2m+1} \}$, and these vertices, and $a$,
are the vertices of a subgraph $W_{2m+2}$.
\end{proof}

\begin{cor}
Let $S \subset \Re^3$ be a finite set with ${\mathcal{D}}_{2m+1} \subseteq \Sym(S)$
for some $m \geq 2$. If $a(S) = 4$ and $G_S$ is a plane graph, then $G_S$ contains a topological
wheel graph $W_{2m+2}$ as a subgraph.
\end{cor}

\begin{prob}
Is it true that if $S \subset \Re^3$ is a finite set
with ${\mathcal{D}}_{2m+1} \subseteq \Sym(S)$ for some $m \geq 2$, and with $a(S) = 4$,
then $G_S$ contains $\mu_p(C_{2t+1})$ as a subgraph, for some $p$ and $t$ satisfying $(2m+1) | (2t+1)$?
If the answer is negative, is it true for V\'azsonyi-critical graphs?
\end{prob}

\section{An additional remark}\label{sec:remarks}

Let $a_k(n)$ denote the maximum of the Borsuk numbers of $n$-dimensional sets of positive diameter,
and let $a(n) = a_1(n)$.
One of the fundamental questions regarding Borsuk's problem is to determine
the asymptotic behavior of $a(n)$.

Presently, the best known asymptotic lower bound for 
$a(n)$ is due to Raigorodskii \cite{R99},
who proved that for sufficiently large values of $n$, 
\[
a(n) \geq \left( \left( \frac{2}{\sqrt3} \right)^{\sqrt2} \right)^{\sqrt{n}}= (1.225...)^{\sqrt{n}}:
\]
he constructed a finite $n$-dimensional set $S$
with the property that the independence 
number of its diameter graph is not greater
than $\frac{\card S}{ 1.225^{\sqrt{n}}}$, if $n$ is sufficiently large.
Clearly, by Remark~\ref{rem:independence}, this property implies not only that 
$a(S) \geq 1.225^{\sqrt{n}}$ for large values of $n$,
but also that $a_{frac}(S) \geq 1.225^{\sqrt{n}}$.
Thus, we have the following.

\begin{rem}
If $n$ is sufficiently large, then for every value of $k$, 
we have $a_k(n) \geq k 1.225^{\sqrt{n}}$.
\end{rem}

\begin{prob}
Is it true that for every value of $k$ and $n$, 
we have $a_k(n) = k a(n)$? If not,
do the two sides have the same magnitude?
\end{prob}

\end{document}